\documentclass[11pt]{amsart}

\usepackage{amssymb,amsmath,amsthm, url}
\usepackage{enumitem}
\usepackage[all]{xy}

\begin{document}

\setenumerate{label=(\alph*), font=\normalfont}

 \newtheorem{thm}{Theorem}[section]
 \newtheorem{prop}[thm]{Proposition}
 \newtheorem{lem}[thm]{Lemma}
 \newtheorem{cor}[thm]{Corollary}
 \newtheorem{conj}[thm]{Conjecture}
\newtheorem{question}[thm]{Question}
 
 \theoremstyle{definition}
 \newtheorem{example}[thm]{Example}
 \newtheorem{defn}[thm]{Definition}
 \newtheorem{remark}[thm]{Remark}

\numberwithin{equation}{section}

\newcommand{\bbN}{{\mathbb{N}}}
\newcommand{\bbZ}{{\mathbb{Z}}}
\newcommand{\bbR}{{\mathbb{R}}}
\newcommand{\bbP}{{\mathbb{P}}}
\newcommand{\bbA}{{\mathbb{A}}}
\newcommand{\bbG}{{\mathbb{G}}}
\newcommand{\bbC}{{\mathbb{C}}}
\newcommand{\bbF}{{\mathbb{F}}}
\newcommand{\bbQ}{{\mathbb{Q}}}
\newcommand{\In}{\operatorname{In}}
\newcommand{\Char}{\operatorname{Char}}
\newcommand{\Cr}{\operatorname{Cr}}
\newcommand{\tr}{\operatorname{Tr}}
\newcommand{\GL}{\operatorname{GL}}
\newcommand{\Mat}{\operatorname{M}}
\newcommand{\SL}{\operatorname{SL}}
\newcommand{\PGL}{\operatorname{PGL}}
\newcommand{\PSL}{\operatorname{PSL}}
\newcommand{\Orth}{\operatorname{O}}
\newcommand{\SOrth}{\operatorname{SO}}
\newcommand{\GO}{\operatorname{GO}}
\newcommand{\PGO}{\operatorname{PGO}}
\newcommand{\POrth}{\operatorname{PO}}
\newcommand{\Stab}{\operatorname{Stab}}
\newcommand{\Aut}{\operatorname{Aut}}
\newcommand{\Hom}{\operatorname{Hom}}
\newcommand{\End}{\operatorname{End}}
\newcommand{\Spec}{\operatorname{Spec}}
\newcommand{\rank}{\operatorname{rank}}
\newcommand{\pr}{\operatorname{pr}}
\newcommand{\id}{\operatorname{id}}
\newcommand{\ed}{\operatorname{ed}}
\newcommand{\pmed}{\operatorname{pmed}}
\newcommand{\Crdim}{\operatorname{Crdim}}
\newcommand{\trdeg}{\operatorname{trdeg}}
\newcommand{\Sym}{\operatorname{S}}
\newcommand{\Alt}{\operatorname{A}}
\newcommand{\Gal}{\operatorname{Gal}}
\newcommand{\Sch}{\operatorname{Sch}}
\newcommand{\im}{\operatorname{im}}
\newcommand{\Id}{\operatorname{Id}}
\newcommand{\Pic}{\operatorname{Pic}}
\newcommand{\Ext}{\operatorname{Ext}}
\newcommand{\calA}{\mathcal{A}}
\newcommand{\calB}{\mathcal{B}}
\newcommand{\calC}{\mathcal{C}}
\newcommand{\calF}{\mathcal{F}}
\newcommand{\calG}{\mathcal{G}}
\newcommand{\calL}{\mathcal{L}}
\newcommand{\calR}{\mathcal{R}}
\newcommand{\calX}{\mathcal{X}}
\newcommand{\calY}{\mathcal{Y}}
\newcommand{\calT}{\mathcal{T}}
\newcommand{\calO}{\mathcal{O}}
\newcommand{\Fields}{\mathbf{Fields}}
\newcommand{\Var}{\mathbf{Var}}
\newcommand{\gVar}{\textrm{-}\mathbf{Var}}
\newcommand{\biratarrow}{\stackrel{\simeq}{\dasharrow}}
\newcommand{\moduli}[2]{\overline{M}_{#1,#2}}
 
\title[Versality and rational points]
{Versality of algebraic group actions
and rational points on twisted varieties}

\author{Alexander Duncan}
\address{Alexander Duncan\newline 
Department of Mathematics, University of Michigan, 
Ann Arbor, MI 48109, USA}
\thanks{A. Duncan was partially supported by
National Science Foundation
RTG grants DMS 0838697 and DMS 0943832.}

\author{Zinovy Reichstein}
\address{Zinovy Reichstein\newline
Department of Mathematics, University of British Columbia,
Vancouver, BC V6T 1Z2, Canada}
\thanks{Z. Reichstein was partially supported by 
National Sciences and Engineering Research Council of
Canada Discovery grant 250217-2012.}


\begin{abstract} 
We formalize and study several competing notions of versality for
an action of a linear algebraic group on an algebraic variety $X$. 
Our main result is that these notions of versality
are equivalent to various statements concerning rational points
on twisted forms of $X$ (existence of rational points, 
existence of a dense set of rational points, etc.) We give 
applications of this equivalence in both directions, to study
versality of group actions and rational points on
algebraic varieties. We obtain similar results
on $p$-versality for a prime integer $p$. An appendix, 
containing a letter from J.-P.~Serre,
puts the notion of versality in a historical perspective.
\end{abstract}

\maketitle

\section{Introduction}
\label{sect1}

Let $k$ be a base field and $G$ be a linear algebraic 
group defined over $k$.  We say that a $G$-action on 
an irreducible $k$-variety $X$ is

\begin{itemize}
\item
{\em weakly versal}, if for every field $K/k$, with $K$ infinite,
and every $G$-torsor
$T \to \Spec(K)$ there is a $G$-equivariant $k$-morphism $f \colon T \to X$,

\item
{\em versal}, if every $G$-invariant dense open subvariety
of $X$ is weakly versal.
\end{itemize}

\noindent
Note that here we view $T$ as a $k$-scheme; it will not be of finite type
in general.  The advantage of the second notion over the first is that 
it only depends on $X$ up to (a $G$-equivariant) birational isomorphism. 
In the case where $X \to B$ is a $G$-torsor
over some irreducible base space $B$,
our definition of versality is identical to~\cite[Definition 5.1]{gms}.
Versal and closely related ``generic" objects (cf.~Remark~\ref{rem.generic})
naturally arise in many parts of algebra and algebraic geometry, such 
as the theory of central simple algebras~\cite{procesi, amitsur, 
saltman-book}, Galois theory~\cite{jly}, and the study of algebraic 
surfaces~\cite{duncan2, tokunaga}.  For a historical perspective 
we refer the reader to the appendix.

Our main result, Theorem~\ref{thm1} below, relates versality 
of $X$ to the existence of $K$-points on certain 
$K$-forms of $X$, for field extensions $K/k$. 
To state it, we need the following additional definitions, 
which will be used throughout the paper.
We will say that a $G$-action on $X$ is

\begin{itemize}
\item
{\em very versal}, if there exists a linear representation $G \to \GL(V)$ 
and a $G$-equivariant dominant rational map $V \dasharrow X$,

\item
{\em birationally linear}, if
there exists a linear representation $G \to \GL(V)$ and a
$G$-equivariant birational isomorphism between $V$ and $X$, 

\item
{\em stably birationally linear}, if there 
exists a linear representation $G \to \GL(W)$ such that  
$X \times W$ is birationally linear.
\end{itemize}

If $K/k$ is a field extension, with $K$ infinite, and 
$\pi \colon T \to \Spec(K)$ 
is a $G$-torsor, we will refer to $(T, K)$ as a {\em twisting pair} 
(see Definition~\ref{def:twistNotation}) and write
$^T X$ for the $K$-variety obtained by twisting $X_K$ by $T$ 
(see Section~\ref{sect.twisting1}).

\begin{thm} \label{thm1}
Let $G$ be a linear algebraic group defined over $k$.
A $G$-action on an irreducible quasiprojective $k$-variety $X$ is
\begin{enumerate}
\item weakly versal if and only if, for every twisting pair $(T, K)$,
\linebreak
$\, ^T X(K) \neq \emptyset$, \label{thm1.a}

\item versal if and only if, for every twisting pair $(T, K)$, $K$-points are
dense in $\, ^T X$, \label{thm1.b}

\item very versal if and only if, for every twisting pair $(T, K)$, 
$\, ^T X$ is $K$-unirational, \label{thm1.c}

\item stably birationally linear if and only if, for every twisting 
pair $(T, K)$, $\, ^T X$ is stably $K$-rational. \label{thm1.d}
\end{enumerate}
\end{thm}

Theorem~\ref{thm1} tells us, in particular, that for a $G$-variety $X$,
the implications
\begin{equation} \label{e.hierarchy}
\text{\small stably birationally linear
$\implies$
very versal 
$\implies$
versal
$\implies$
weakly versal} 
\end{equation}
cannot be reversed in general, even if $G = \{ 1 \}$. Simple examples of versal 
$G$-varieties $X$ that are not very versal can be constructed as follows.
Suppose $X$ is versal. Then by the definition of versality
$X \times_k Y$ is also versal, for any variety $Y$ 
with trivial $G$-action, provided that $k$-points are dense in $Y$.  
If $Y$ is not unirational (e.g., is an elliptic curve), then 
$X \times_k Y$ is not unirational and hence, cannot be very versal.  

On the other hand, in many natural examples $X$ is geometrically 
unirational, i.e., $X$ becomes unirational over an algebraic closure $\bar{k}$.
In this situation  the twisted $K$-variety $^T X$ is geometrically 
unirational for every twisting pair $(T, K)$.  For a smooth 
geometrically unirational variety $Y$ defined over an infinite
field $K$, it is not known whether or not the following properties 
are equivalent:
(i) $Y$ is $K$-unirational, (ii) $K$-points are dense in $Y$, 
and (iii) $Y$ has a $K$-point; see~\cite[Question 1.3]{kollar}.
It is thus conceivable that if $X$ is smooth and geometrically 
unirational then the second and third implications in~\eqref{e.hierarchy}
may, indeed, be reversed. This explains why part (b) of Theorem~\ref{thm1},
which takes the most delicate arguments to prove, is never truly used in 
the specific examples in this paper, i.e., why versal varieties 
in these examples turn out to be very versal.

The rest of this paper is structured as follows. Section~\ref{sect.prel}
is devoted to notation and preliminaries, and
Section~\ref{sect.twisting1} to a discussion 
of the twisting operation. In
Sections~\ref{sect.proof-of-thm1ab} and~\ref{sect.proof-of-thm1cd} we
prove Theorem~\ref{thm1}.
In Section~\ref{sec:ModuliSpaces} we 
show that every $K$-form of the moduli spaces
$\moduli{0}{n}$ and $\moduli{1}{n}$
are unirational over $K$ for certain $n$.
In Section~\ref{sect.homogSpace} we use Theorem~\ref{thm1} in the other
direction to give a versality criterion for the action 
of a closed subgroup $G \subset A$ on a homogeneous 
space $A/B$. 
In Section~\ref{sect.p-versality} we define and study the related 
notions of $p$-versality, where $p$ is a prime integer.
We show that $p$-versality is related to $0$-cycles on
twisted varieties (rather than points) and that
for smooth varieties, weak $p$-versality is
equivalent to $p$-versality.
Sections~\ref{sect.proj-rep} and~\ref{sect.quadric-cubic} feature 
versality criteria for group actions on projective spaces and
low degree hypersurfaces.  As an application, we show that 
a recent conjecture of I.~Dolgachev
on the Cremona dimension is incompatible with a long-standing 
conjecture of J.~W.~S.~Cassels and P.~Swinnerton-Dyer about rational points 
on cubic hypersurfaces.

\section{Notation and preliminaries}
\label{sect.prel}

Let $k$ be a field; we will denote an algebraic closure of $k$
by $\bar{k}$.

A \emph{$k$-variety $X$} is a reduced, quasiprojective scheme 
of finite type over $k$, not necessarily irreducible.
A \emph{morphism of $k$-varieties} 
is a morphism of schemes respecting the structure morphism to $k$.
A \emph{product of $k$-varieties} is a fibre product
of schemes over $k$.

An \emph{algebraic group $G$ over $k$} is a smooth affine group scheme 
of finite type over $k$.  An \emph{action of $G$ on $X$} 
will always be a morphic action, i.e., 
a morphism $\sigma: G \times X \to X$ satisfying 
the standard conditions \cite{GIT}. We will sometimes refer to
$X$ with an action of $G$ as a $G$-variety.

Given a $k$-variety $X$ and a field extension $K/k$,
the symbol $X_K$ denotes the $K$-variety $X \times_{\Spec(k)} \Spec(K)$.
A \emph{$k$-form of $X$} is a $k$-variety $X'$ such that
$X_{\bar{k}} \simeq X'_{\bar{k}}$.

A \emph{right} (resp. \emph{left}) \emph{$G$-torsor over $Y$} 
is a flat morphism $\psi : X \to Y$ of $k$-schemes such that 
$G$ acts on $X$ on the right (resp. left), 
and the map $\sigma \times \pi_2: G_Y \times_Y X \to X \times_Y X$
is an isomorphism (where $\sigma$ is the action map and $\pi_2$ is the
second projection).
The set of $G$-torsors over 
a field $K$ is in bijection with the Galois cohomology set $H^1(K,G)$.

A $k$-variety is \emph{rational} if it is $k$-birationally equivalent to
$\bbA^n$, for some positive integer $n$. A $k$-variety $X$ is
\emph{unirational} if there exists a dominant rational $k$-map
$\bbA^n \dasharrow X$.

A $G$-action on $X$ is \emph{generically free}
if there exists a dense $G$-invariant open subvariety $X_0 \subset X$
such that the scheme-theoretic stabilizer of every point $x \in X_0$
is trivial.
This is equivalent to the existence of a dense $G$-invariant open 
subvariety $U$ of $X$ which is the total space of a $G$-torsor
$\pi: U \to B$; see 
\cite[Theorem 4.7]{BerhuyFavi}.  If $B$ is irreducible, 
we say that $X$ is a \emph{primitive} $G$-variety.

Given a generically free primitive $G$-variety $X$ one obtains a
$G$-torsor over $\Spec(k(B))$ by taking the generic fibre of $\pi$.
Conversely, given a $G$-torsor over a finitely generated field extension
$K/k$, one can recover a birational equivalence class of
generically free primitive $G$-varieties for
which $k(B)=K$.
Indeed, the $G$-torsor is an affine $K$-variety and
thus can be defined over some finitely generated $k$-subalgebra of $K$.

The following remark is an immediate consequence of this correspondence.

\begin{remark} \label{rem.gen-free}
A $G$-action on an irreducible $k$-variety $X$ is 
weakly versal if and only if, for every generically 
free primitive $G$-variety
$Y$ defined over $k$, with $k(Y)^G$ infinite,
there exists a $G$-equivariant $k$-rational 
map $Y \dasharrow X$.
\end{remark}

\begin{prop} \label{prop.weakly-versal}
Let $X$ be an irreducible $G$-variety defined over $k$. If $G$ has a fixed
$k$-point $x \in X(k)$ then $X$ is weakly versal.
\end{prop}

\begin{proof}
For every field $K/k$ and every $G$-torsor $T \to \Spec(K)$,
the constant map $T \to X$, sending all of $T$ to $x$, is
$G$-equivariant.
\end{proof}

We will view a linear representation $G \to \GL(V)$ of $G$ on 
a $k$-vector space $V$ as a $G$-action on the affine space associated to $V$.
By abuse of notation, we will denote this affine space by the same 
symbol $V$.  We will say that a linear representation $G \to \GL(V)$ is
generically free if the $G$-variety $V$ is generically free.
The $G$-torsors $U \to B$, associated to generically free linear
representations $G \to \GL(V)$, as above, will play an important 
role in this paper. Note that torsors of this form may be viewed as 
``algebraic approximations" to the universal torsor over 
the classifying space of $G$, in the sense of B.~Totaro~\cite{totaro}.

\begin{remark} \label{rem.gf-rep}
Let $G$ be an algebraic group over $k$. Then there exists a generically 
free linear representation $G \to \GL(V)$;
see~\cite[Remark 4.12]{BerhuyFavi}. Moreover,
adding a copy of the trivial representation if necessary,
we can choose $V$ so that $k(V)^G$ is an infinite field.
\end{remark}

\begin{remark}
It is vital that algebraic groups are \emph{linear} throughout this paper.
For example, if $A \ne \{ 1 \}$ is an abelian variety then 
no versal $A$-torsor can
exist.  Otherwise the exponent of every element of $H^1(K, A)$, for any
field $K/k$, would divide the exponent of the versal torsor. 
However, there are $A$-torsors of arbitrarily high exponent 
(over suitable $K$); cf.~\cite[Section 3]{brosnan}.
\end{remark}

\begin{prop} \label{prop:versalGeomIrred}
If $X$ is a versal irreducible $G$-variety then $X$ 
is geometrically irreducible.
\end{prop}

\begin{proof}
It suffices to show that $X_{k^s}$ is irreducible, 
where $k^s$ denotes a separable closure of $k$;
see~\cite[Exercise 2.3.15(a)]{hartshorne}.
Let $X_1, \dots, X_n$ denote the irreducible components of 
$X_{k^s}$.  We want to show that $n = 1$.
We will argue by contradiction.  Assume $n \geqslant 2$. 
Since $X$ is irreducible over $k$, the absolute Galois group
$\Gal(k)$ permutes $X_1, \dots, X_n$
transitively. Thus $Y:= X_1 \cap \dots \cap X_n$
is a closed $G$-invariant $k$-subvariety of $X$ (possibly empty)
and $X \ne Y$.

Let $V$ be a generically free linear $G$-representation with $k(V)^G$ infinite.
By Remark~\ref{rem.gen-free}
there exists a $G$-equivariant rational
$k$-map $f \colon V \dasharrow X \setminus Y$.
Since $V$ is geometrically irreducible, the image of $f$ is 
contained in one of the components $X_i$.
Since $\Gal(k)$ transitively permutes the components,
the image of $f$ is also contained
in $X_2, \dots, X_n$ and thus in $Y$, a contradiction.
\end{proof}

\begin{remark} \label{rem.minimal}
Suppose $X$ is of minimal dimension among generically free 
versal $G$-varieties. Then $X$ must be very versal.  
To see this, let $V$ be a generically free 
linear representation of $G$, as in Remark~\ref{rem.gf-rep}.
Let $U \subset X$ be a $G$-invariant open subvariety which is the
total space of a $G$-torsor $U \to B$.
Then $U$ is weakly versal and,
by Remark~\ref{rem.gen-free},
there exists a $G$-equivariant rational 
map $f: V \dasharrow U$.  
The closure $Z$ of the image of $f$ is very versal and, since $U$
is a $G$-torsor, the action on $Z$ is generically free.
By minimality of $\dim(X)$,
we conclude that $\dim(Z) = \dim(X)$ and thus $f$ is dominant.

The minimal value of $\dim(X) - \dim(G)$, as $X$ ranges over the very versal
(or equivalently, versal) generically free $G$-varieties,
is called the {\em essential dimension} of $G$ and is denoted by $\ed(G)$; 
see~\cite{BerhuyFavi},~\cite[Section 5]{gms} or~\cite{icm}. 
\end{remark}

\begin{remark} \label{rem.fixed-pt} 
The following criterion gives a convenient way (often the only
known way) to show that a given action is not versal.

{\em Let $X$ be a projective irreducible
weakly versal $G$-variety defined over $k$. Suppose $H \subset G$ 
is a closed $k$-subgroup such that every 
finite-dimensional $k$-representation of $H$ has a $1$-dimensional
invariant subspace. Then $X$ has an $H$-fixed $k$-point.} 

In applications, $H$ is usually taken to be a 
finite constant abelian group. When $k$ has suitable roots of unity,
every linear representation of such a group decomposes as a direct sum 
of $1$-dimensional character spaces; in particular, the above 
criterion applies. For example, 
the usual action of the alternating group $G = \Alt_5$ on $X = \bbP^1$ (say, 
over the field $k = \bbC$ of complex numbers) is not versal. Indeed, 
$G$ contains a subgroup $H$ isomorphic to 
$\bbZ/2 \bbZ \times \bbZ/2\bbZ$ which has no fixed points in
$\bbP^1$. Here one generator of $\bbZ/2 \bbZ \times \bbZ/2\bbZ$ 
takes $(x:y) \in \bbP^1$ to $(-x : y)$ and the other to $(y: x)$.
Arguments of this type are used extensively, e.g., in~\cite{duncan2},
\cite{duncan},~\cite{BeauvilleED} and \cite{tokunaga}.

To prove the criterion, let $V$ be a generically free $G$-representation.
By Remark~\ref{rem.gen-free}, there is a $G$-equivariant 
(and hence, $H$-equivariant) rational map  $V \dasharrow X$.
Since $V$ has a smooth $H$-fixed point (the origin), 
the ``Going Down Theorem'' implies that $X$ has an $H$-fixed point; 
see~\cite[Proposition A.2 and Remark A.7]{ry-ed}.
\qed
\end{remark}

\begin{remark} \label{rem.generic}
The notion of a {\em generic} group action, in the sense of
D.~J.~Saltman~\cite{saltman-generic}, 
is closely related to our notion of versality but is not 
identical to it. The action of $G$ on an irreducible 
variety $X$ is called generic if it is generically free, versal 
and $k(X)^G$ is a rational (i.e., purely transcendental) field extension 
of $k$.  Note that while versal actions exist for every algebraic group $G$,
generic $G$-actions may or may not exist, depending on $G$.
We also note that for a finite 
group $G$ the existence of
a generic action is equivalent to the existence of a generic polynomial 
in the sense of F.~DeMeyer~\cite{demeyer, jly}. The distinction
between ``versal" and ``generic" is often blurred in the literature. 
\end{remark}

\section{Twisting}
\label{sect.twisting1}

Let $G/k$ be an algebraic group, $X/k$ be a $G$-variety, 
and $T \to \Spec(k)$ be a right $G$-torsor.
The diagonal action of $G$ on $T \times X$ makes $T \times X$ into
the total space of a $G$-torsor $T \times X \to B$. The base space $B$
of this torsor is unique (it is the geometric quotient of $T \times X$ 
by $G$); it is usually called {\em the twist} of $X$ by $T$ and denoted 
by $^T X$.  This construction relies
on our standing assumption that $X$ is quasiprojective;
for details, see~\cite[Section 2]{florence2} or~\cite[Section 2]{ckpr}.

Note that there is no natural $G$-action on $^TX$; 
we lose the $G$-action in the course of constructing
$^T X$.  However, $^T X$ carries a natural action 
of the twisted group $^T G$; see 
Propositions~\ref{prop:twistedGroup} and~\ref{prop:TGaction} below.

If $T$ is split over $k$, it is easy to see that $^T X$ 
is $k$-isomorphic to $X$.  Hence, $^T X$ is 
a $k$-form of $X$, i.e., $X$ and $^T X$ become isomorphic 
over any splitting field $L/k$ of $T$. Combining this observation with
Hilbert's Theorem 90 (\cite[Proposition X.1.3]{serre-lf}), we obtain
the following.

\begin{lem} \label{lem.Hilbert90} Let $V$ be a linear representation
$G \to \GL(V)$ and $T \to \Spec(k)$ be a $G$-torsor.
Then $^T V$ is $k$-isomorphic to $V$.  In particular, $^T V(k)$ is dense 
in $^T V$.
\qed
\end{lem}

It is well-known that quasiprojectivity is a geometric property,
in the sense that if a $k$-variety is quasiprojective over $\bar{k}$ 
then it is quasiprojective over $k$. Thus, twisting is performed entirely  
within the category of quasiprojective varieties.

Twisting is functorial in the following sense: a $G$-equivariant 
morphism $f \colon X \to Y$
(respectively, a rational map $f \colon X \dasharrow Y$) 
of $G$-varieties gives rise to a $K$-morphism 
$^T f \colon {}^T X \to {}^T Y$ (respectively, a $K$-rational map 
$^T f \colon {}^T X \dasharrow {}^T Y$). For details, 
see~\cite[Lemma 2.2]{florence2} (where only rational maps 
are considered, but the construction of $^T f$ is even more 
straightforward if $f$ is regular).

The following Proposition amplifies~\cite[Lemma 3.4]{ckpr}.

\begin{prop} \label{prop:twistAdjunction}
Let $k$ be a field, $G$ be an algebraic group
and $T \to \Spec(k)$ be a $G$-torsor.
Denote by $\Var$ the category of $k$-varieties and
by $G\gVar$ the category of $k$-varieties with a $G$-action. 
Morphisms in the latter category are $G$-equivariant $k$-maps.

Let $\calL_T : \Var \to G\gVar$ be the functor $T \times -$ which 
takes a $k$-variety $Y$ to $T \times Y$ (viewed as a $G$-variety,
with $G$ acting trivially on $Y$). 
Let $\calR_T : G\gVar \to \Var$ be the twisting functor $(^T -)$ 
described above.

Then the functors $(\calL_T,\calR_T)$ form an adjoint pair.  In other 
words, for any $Y \in \Var$ and $X \in G\text{-}\Var$, we have 
an isomorphism
\begin{equation}
\Hom_{G\gVar}(T \times Y, X) \simeq \Hom_{\Var}(Y ,{}^T X)
\end{equation}
which is functorial in both $X$ and $Y$.
\end{prop}

\begin{proof}
The isomorphism is easiest to see by considering the intermediate set
$\calF(Y,X)$ consisting of
$G$-equivariant morphisms $\gamma : T \times Y \to T \times X$
such that the following diagram commutes:
\[ \xymatrix{
T \times Y \ar[r]^{\gamma} \ar[rd]_{\pr_1}&
T \times X \ar[d]^{\pr_1} \\
 & T} \]
where the vertical maps are projections.

Given $\gamma \in \calF(Y,X)$, we obtain $\alpha = \pr_2 \circ \gamma$
in $\Hom_{G\gVar}(T \times Y, X)$.
Mapping $\alpha \in \Hom_{G\gVar}(T \times Y, X)$
to $\gamma = \pr_1 \times \alpha$ is an inverse.
Given $\gamma \in \calF(Y,X)$, we obtain
$\beta \in \Hom_{\Var}(Y ,{}^T X)$ by taking quotients.
All of these operations are clearly functorial.

It remains to reconstruct
$\gamma \in \calF(Y,X)$ given $\beta: Y \to {}^T X$.
Pulling back the $G$-torsor
$T \times X \to {}^T X$ we obtain a $G$-torsor $\pi : Y' \to Y$:
\[ \xymatrix{  Y' \ar[d]_{\text{ \tiny   $G$-torsor}} 
\ar[r]^{f} & \; T \times X \ar[d]^{\text{ \tiny   $G$-torsor}} \\
               Y \ar[r] \ar[r]^{\beta}     &  ^T X . } \]
The $G$-equivariant map $\phi = (\pr_1 \circ f) \times \pi$
is a morphism of $G$-torsors $\phi: Y' \to T \times Y$ over $Y$.
By a standard result on torsors, this means $\phi$ is an isomorphism.

Thus, we have a $G$-equivariant morphism
$\gamma' = f \circ \phi^{-1} : T \times Y \to T \times X$ which lifts $\beta$.
However, since pullbacks are only defined up to isomorphism,
the projections $T \times Y \to T$ and $T \times X \to T$ do not
necessarily commute with $\gamma'$.  Nevertheless, there exists a unique $g
\in G$ such that $g \circ \gamma'$ is in $\calF(Y,X)$.
This is the desired $\gamma$.  The construction is easily seen to be
functorial.
\end{proof}

\begin{cor} \label{cor:adjunctionSplitPoints}
Let $X$ be a $G$-variety, and $T \to \Spec(k)$ be a $G$-torsor. Let $L/k$
be a splitting field of $T$, let $s$ be a point in $T(L)$, and let $t_s :
({}^TX)_L \to X_L$ be the $L$-isomorphism such that
$s \times t_s$ is a section of 
$\pi_L : T_L \times X_L \to ({}^TX)_L$. 
If $\alpha \colon T \to X$ and $\beta \colon
\Spec(k) \to {}^TX$ are corresponding maps under the adjunction of
Proposition~\ref{prop:twistAdjunction} {\rm (}with $Y = \Spec(k)${\rm )}
then $\alpha_L(s) = t_s(\beta_L)$ in $X(L)$.
\end{cor}

\begin{proof}
By definition, $\alpha$ and $\beta$ fit into the following 
commutative diagram of $G$-equivariant $k$-morphisms:
\[ \xymatrix{ T \ar@{->}[r]^{\id \times \alpha} \ar@{->}[d]  
& T \times X \ar@{->}[d] \\   
\Spec(k)  \ar@{->}[r]^{\beta}  
& ^T X \; . } \]   
The vertical maps are $G$-torsors; we split them by base-changing
from $k$ to $L$.
By the definition of $t_s$ the resulting diagram 
\[ \xymatrix{ T_L \ar@{->}[r]^{\id \times \alpha_L} \ar@{->}[d]  
& T_L \times X_L \ar@{->}[d] \\   
\Spec(L)  \ar@{->}[r]^{\beta_L}  \ar@{->}@/^1.5pc/[u]^s 
& (^T X)_L \ar@{->}@/_1.5pc/[u]_{s \times t_s} \; . } \]   
is commutative.
Tracing from the lower left corner to the upper right, we see that
$s \times \alpha_L(s) = s \times t_s (\beta_L)$ as morphisms
$\Spec(L) \to T_L \times X_L$. Composing these morphisms with the projection 
$T_L \times X_L \to X_L$, we see that
$\alpha_L(s) = t_s(\beta_L)$ as maps $\Spec(L) \to X$.
\end{proof}

\begin{cor}\label{cor:twistsPreserve}
Let $X$ and $Y$ be $G$-varieties defined over $k$, and let $T \to \Spec(k)$
be a $G$-torsor.  Then 
\begin{enumerate}
\item ${}^T (X \times Y)$ is canonically isomorphic to ${}^T X \times {}^T Y$.
\label{cor:twistsPreserve:prod}

\item Let $f \colon X \to Y$ be a $G$-equivariant closed 
(resp. open) immersion.  Then $^T f \colon {}^T X \to {}^T Y$
is also a closed (resp. open) immersion.
\label{cor:twistsPreserve:top}

\item If $f \colon X \dasharrow Y$ is a $G$-equivariant dominant  
rational map then the induced rational map 
$^T f \colon {}^T X \dasharrow {}^T Y$ is also dominant.
\label{cor:twistsPreserve:dominant}
\end{enumerate}
\end{cor}

\begin{proof}
\ref{cor:twistsPreserve:prod} follows from the fact that 
the twisting functor is a right adjoint and,
thus, is left exact. 
To prove
\ref{cor:twistsPreserve:top} and
\ref{cor:twistsPreserve:dominant} note
that by \cite[Proposition 2.7.1]{EGA4}
the properties of being a closed or open immersion 
and of being dominant 
are geometric.  In other words, for the
purpose of checking that $^T f$ has these properties, we may pass 
to any field extension $L/k$. 
In particular, we may replace $k$ by a splitting field of $T$
and thus assume without loss of generality that $T \to \Spec(k)$ is split.
In this case $^T X$, $^T Y$ and $^T f$
become $X$, $Y$, and $f$, respectively,
and the assertions of parts \ref{cor:twistsPreserve:top} and
\ref{cor:twistsPreserve:dominant} become obvious.
\end{proof}

\begin{prop} {\rm (}cf.~\cite[Lemma 3.5]{ckpr}{\rm )}
\label{prop:twistedGroup}
Suppose $H$ and $G$ are algebraic groups over $k$,
and $G$ acts on $H$ by group automorphisms.  
Let $T \to \Spec(k)$ be a $G$-torsor.  
Then $^T H$ is a $k$-form of the algebraic group
$H$. In particular, 
$^T H$ is an affine algebraic group over $k$.
\end{prop}

\begin{proof}
The commutative diagrams defining 
the group scheme structure on
$H$ are all $G$-equivariant.
Applying the twisting functor to these diagrams, and 
using Corollary \ref{cor:twistsPreserve}(a),
we see that $^T H$ is an algebraic group.
If $L/k$ is a splitting field of $T$ then clearly
$(^T H)_L$ and $H_L$ are isomorphic.
The assertion
that $H$ is affine follows by descent; 
see~\cite[Proposition~2.7.1(xiii)]{EGA4}.
\end{proof}

\begin{prop} \label{prop:TGaction}
Let $T \to \Spec(k)$ be a $G$-torsor and let ${}^T G$ 
denote the twist by $T$ of the conjugation action of $G$ on itself.  
For every $G$-variety $X$, the $G$-action on $X$ induces 
a ${}^T G$-action on ${}^T X$.  
Moreover, for every $G$-equivariant morphism $f$ 
the morphism $^T f$ is $^T G$ equivariant.  In other words, 
the twisting functor factors through the category of ${}^TG$-varieties.
\end{prop}

\begin{proof}
The action map $G \times X \to X$ and associated commutative 
diagrams are all $G$-equivariant. 
As in the proof of Proposition~\ref{prop:twistedGroup}, we obtain an
action map $^T G \times {}^T X \to {}^T X$ and commutative
diagrams which show that $^T f$ is $^T G$ equivariant.  
\end{proof}

\section{Proof of Theorem~\ref{thm1}(a) and (b)}
\label{sect.proof-of-thm1ab}

We will use repeatedly the fact that twisting commutes with base 
field extension. Given a $k$-variety $X$, a field extension $K/k$, 
and a $G$-torsor $T \to \Spec(K)$, we will use the shorthand 
notation ${}^T X$ to denote ${}^T (X_K)$.

For brevity, we use the following terminology throughout the paper.

\begin{defn} \label{def:twistNotation}
Let $k$ be a field and $G$ be an algebraic group over $k$.
By a {\em $G$-twisting
pair} $(T, K)$ we shall mean a choice of a field extension $K/k$, with $K$
infinite, and a $G$-torsor $T \to \Spec(K)$. In situations where the choice of
$G$ is clear from the context and there is no risk of ambiguity, we will
simply refer to $(T, K)$ as a \emph{twisting pair}.
\end{defn}

\begin{proof}[Proof of Theorem~\ref{thm1}\ref{thm1.a}]
Let $(T, K)$ be any twisting pair.
Setting $Y = \Spec(K)$ in Proposition~\ref{prop:twistAdjunction}, 
we see that the $K$-points of $^T X$
are in a natural $1-1$ correspondence with $G$-equivariant maps $T \to X$.
\end{proof}

The proof of Theorem~\ref{thm1}\ref{thm1.b} is considerably more 
delicate. Before we proceed with the details, we would like 
to explain a new obstacle our argument will have to overcome.

Given a $G$-action on an irreducible $X$, and a $G$-twisting pair $(T, K)$,
let us say that $X$ is $(T, K)$-weakly versal if there exists
a morphism $T \to X$ defined over $k$. The $G$-action on $X$ is, 
by definition, weakly versal if it is $(T, K)$-weakly versal for
every twisting pair $(T, K)$.  Note that our proof of
Theorem~\ref{thm1}\ref{thm1.a} establishes the following stronger statement:
 
\smallskip
{\em Choose a $G$-twisting pair $(T, K)$. Then an irreducible $G$-variety $X$ 
is $(T, K)$-weakly versal if and only if $^T X$ has a $K$-point.}

\smallskip
Similarly, given a $G$-twisting pair $(T, K)$, we will say that an
irreducible $G$-variety $X$ is $(T, K)$-versal
if every dense $G$-invariant open
subvariety of $X$ is $(T, K)$-weakly versal. One is thus naturally led to try
to prove Theorem~\ref{thm1}\ref{thm1.b} by showing that for any given
$G$-twisting pair $(T, K)$, $X$ is $(T, K)$-versal if and only if $K$-points
are dense in $^T X$. The following example shows that this stronger version
of Theorem~\ref{thm1}(b) fails.

\begin{example} \label{ex.motivation} 
Let $k = \mathbb{C}$ and let $X$ be a smooth irreducible 
projective complex curve of genus $g \geqslant 2$, whose 
automorphism group $G := \Aut(X)$ is non-trivial.  
By Hurwitz's theorem, $G$ is finite.

Let $\pi \colon X \to X/G$ be the quotient map, 
let $K := k(X)^G = k(X/G)$, and let $T \to \Spec(K)$ be 
the $G$-torsor obtained by pulling back $\pi$
via the generic point $\Spec(K) \to X/G$. Then
the $G$-action on $X$ is $(T, K)$-versal, since
the identity map $X \to X$ restricts to 
a $G$-equivariant morphism $T \to U$ for 
any $G$-invariant open subvariety $U \subset X$.

On the other hand, we claim that the $K$-curve $^T X$ has 
only finitely many $K$-points, and hence,
$K$-points cannot be dense in $^T X$. Indeed,
arguing as in the proof of Theorem~\ref{thm1}\ref{thm1.a},
we see that $K$-points of $^T X$ are in a natural 
bijective correspondence with $G$-equivariant 
$k$-morphisms $T \to X$ or equivalently,
with $G$-equivariant rational maps $X \dasharrow X$,
or equivalently (since $X$ is a smooth complete curve)
with $G$-equivariant morphisms $X \to X$.
The latter can be of two types: (i) dominant
and (ii) constant (i.e., the image is a single point of $X$).
It thus suffices to show that there are only finitely many 
morphisms $X \to X$ of each type.

(i) Since $g \geqslant 2$, the Hurwitz formula tells us that any dominant 
morphism $X \to X$ is, in fact, an automorphism of $X$. As we mentioned above,
$X$ has only finitely many automorphisms. 

(ii) If the image of $f$ is a point of $X$, 
this point has to be fixed by $G$, and $X$ has only
finitely many $G$-fixed points. This completes the proof of the claim.
\qed
\end{example}

The above example demonstrates that, given a twisting pair $(T, K)$,
we cannot hope to deduce the density of $K$-points in $^T X$ merely from 
the fact that $X$ is $(T, K)$-versal.
When $X$ is versal, we will deduce 
the density of $K$-points in $^T X$, for every 
twisting pair $(T, K)$, from the fact 
that, in particular, $X$ is $(S, F)$-versal, where $S$ and $F$ are as follows.

\begin{defn} \label{def.vector-space} 
For the rest of this section and in Section~\ref{sect.proof-of-thm1cd}:

\begin{itemize}
\smallskip
\item
$V$ will denote a generically free linear representation of $G$.

\smallskip
\item
$F$ will denote the field $k(V)^G$. We will choose $V$ so that
$F$ is infinite (see Remark~\ref{rem.gf-rep}).

\smallskip
\item
$V_0$ will denote a dense open $G$-invariant subvariety of $V$ which is the 
total space of a $G$-torsor $V_0 \to B$.

\smallskip
\item
$S \to \Spec(F)$ will denote the $G$-torsor obtained by
pulling back $V_0 \to B$ via the generic point $\eta \colon \Spec(F) \to B$.
\end{itemize}
\end{defn}

We now proceed with the proof of Theorem~\ref{thm1}\ref{thm1.b}.

\begin{lem} \label{lem:canHitAnything}
Let $X/k$ be a geometrically irreducible $G$-variety,
and suppose $X$ is $(T, K)$-versal
for some twisting pair $(T, K)$.  Then, for any field extension
$L/k$ and for any closed $L$-subvariety $Y \subsetneq X_L$
(not necessarily $G$-invariant),
there exists a $G$-equivariant $k$-morphism $\psi: T \to X$ 
such that the image of
$\psi_L : T \times_k \Spec(L) \to X_L$ is not contained in $Y$.
\end{lem}

\begin{proof}
First assume $L = k$. If $Y$ is $G$-invariant, the lemma follows from 
the definition of $(T, K)$-versality. If $Y$ is not 
$G$-invariant, we proceed as follows.  Let $Z$ be the closure of the union 
$\bigcup \im(\psi)$, where
$\psi : T \to X$ varies over all $G$-equivariant $k$-morphisms whose image is
contained in $Y$.  Since each $\psi$ is $G$-equivariant,
the closure of each $\im(\psi)$ is $G$-invariant,
as is the closure of their union.
In other words, the subvariety $Z$ is $G$-invariant.
Note that $Z \subseteq Y \subsetneq X$.
Since $X$ is $(T, K)$-versal, there is a $k$-morphism $\psi : T \to X$
whose image is in the complement of $Z$.  By the construction of $Z$, the
image of any such map is not contained in $Y$.  This completes
the proof of the lemma in the case where $k = L$.

Now assume $L/k$ is arbitrary. 
Let $X = U_1 \cup \dots \cup U_m$ be an open affine cover of $X$
defined over $k$. (We do not
assume that the $U_i$ are $G$-invariant.) The defining equations of $Y$ in
each $U_i$ involve only a finite number of elements of $L$.  Let $R$ be the
$k$-subalgebra of $L$ generated by all these elements.  Then $Y$ is, in fact,
defined over $\Spec(R)$.  In other words, there exists a closed $k$-subvariety
$Y_0 \subset X_R = X \times_k \Spec(R)$ such that $Y = Y_0 \times_R L$.

Since $Y \neq X_L$, clearly $Y_0 \neq X_R$.  Let 
$\pi: X_R = X \times_k \Spec(R) \to X$ be the natural projection and 
\[ C := \{ x \in X \, | \, \pi^{-1}(x) \subset Y_0 \} . \]
Since $\pi$ is an open map and $C$ is the complement of the image of
the complement of $Y_0$, $C$ is a closed subvariety of $X$  
defined over $k$ and $C \ne X$ (because $Y_0 \neq X_R$).
As we showed above, there is a $G$-equivariant $k$-morphism 
$\psi : T \to X$ whose image is not contained in $C$.  
Then the image of $\psi \times_k \Spec(R) \colon T \times \Spec(R) \to
X_R$ is not contained in $Y_0$ 
and thus, the image of $\psi_L \colon T \times \Spec(L) \to X_L$ 
is not contained in $Y$.
\end{proof}

\begin{cor} \label{cor:sepDense}
Let $X$ be a geometrically irreducible $k$-variety,
and let $L/k$ be a field extension.
Note that there is a natural inclusion of 
sets $X(k) \hookrightarrow X_L(L)$ by pulling back
$\Spec(k) \to X$ by $\Spec(L) \to \Spec(k)$.
Then $X(k)$ is dense in $X$ if and only if $X(k)$ is dense in $X_L$.
\end{cor}

\begin{proof}
If $L$ is finite then the result is immediate: in this case $X(k)$ 
is dense in $X$ if and only if $X(k)$ is dense in $X(L)$ if and only if
$X$ is a point. Thus we may assume that $L$ is infinite.
The morphism $X_L \to X$ is dominant, so one direction is obvious.
The other implication is a special case of
Lemma~\ref{lem:canHitAnything} with $K=L$, $G= \{ 1 \}$ and $T=\Spec(K)$.
\end{proof}

\begin{lem} \label{lem:arithVersalEquiv}
Let $X/k$ be a geometrically irreducible $G$-variety,
let $(T, K)$ be a twisting pair, 
and let $L/K$ be a field extension which splits $T$.
Fix an $L$-point $s \in T(L)$.  Then the following are equivalent:
\begin{enumerate}
\item $(^TX)(K)$ is dense in $^TX$, \label{lem:arith:twists}

\item the set of points $f_L(s)$, where $f$ varies over all
$G$-equivariant $k$-morphisms $f \colon T \to X$, 
is dense in $X_L$.
\label{lem:arith:points}
\end{enumerate}
\end{lem}

Note that condition \ref{lem:arith:points} is considerably stronger
than the condition 
that the union of $f_L(T)$ is dense in $X_L$, which came up in
Lemma~\ref{lem:canHitAnything}.  This discrepancy is precisely 
the source of the difficulty we encountered in
Example~\ref{ex.motivation}.

\begin{proof}
Since $X$ is geometrically irreducible, so is $^TX$.
By Corollary~\ref{cor:sepDense}, 
condition \ref{lem:arith:twists} is equivalent to 
\begin{enumerate}[start=3]
\item $(^TX)(K)$ is dense in $({}^TX)_L$. \label{lem:arith:temp1}
\end{enumerate}
\noindent
Let $t_s$ be an $L$-isomorphism between $(\, ^T X)_L$ and $X_L$,
chosen so that 
\[ (s, t_s) \colon \,  (^T X)_L \to T_L \times X_L \]
is a section (defined over $L$) of the $G$-torsor 
$T \times X \to \, ^T X$, as in
the statement of Corollary~\ref{cor:adjunctionSplitPoints}.
Then \ref{lem:arith:temp1} is equivalent to 
\begin{enumerate}[start=4]
\item the set of $L$-points of the form $t_s(q)$,
where $q$ varies over $(^TX)(K)$, is dense in $X_L$. \label{lem:arith:temp2}
\end{enumerate}
\noindent
By Corollary~\ref{cor:adjunctionSplitPoints}, \ref{lem:arith:temp2} is
equivalent to \ref{lem:arith:points}.
\end{proof}

\begin{proof}[Proof of Theorem~\ref{thm1}\ref{thm1.b}]
$\impliedby$ (cf.~\cite[Proposition 1.12]{faviflorence}):
Assume $K$-points are dense in $^T X$ for every
twisting pair $(T, K)$. We want to
show that every dense $G$-invariant open subvariety $U \subset X$ is weakly
versal. By Theorem~\ref{thm1}\ref{thm1.a} it suffices to show that $^T U$
contains a $K$-point for every twisting pair $(T, K)$, as above.
This follows from the fact that $^T U$ is a dense open subvariety of $^T X$;
see Corollary~\ref{cor:twistsPreserve}\ref{cor:twistsPreserve:top}.

$\implies$:
Assume $X$ is versal.  Then $X$ is geometrically irreducible; see
Proposition~\ref{prop:versalGeomIrred}.
Fix a twisting pair $(T, K)$.
We want to show $K$-points are dense in $^TX$.
Let $L$ be a splitting field for $T$, and let $s$ be a point in $T(L)$.
By Lemma~\ref{lem:arithVersalEquiv}, it suffices to show
that for every closed subvariety $Y \subsetneq X_L$ defined over $L$,
there exists a $G$-equivariant $k$-morphism $f \colon T \to X$ 
such that $f_L(s) \not \in Y$.

As explained above, we cannot construct $f$ directly
using only the fact that $X$ is $(T, K)$-versal. We will 
instead construct $f$ in two steps, as a composition of 
a $G$-equivariant $k$-morphism $f_1 \colon T \to V$ and 
a $G$-equivariant rational map $f_2 \colon V \dasharrow X$. 
Here $V$, $S$ and $F$ are as in Definition~\ref{def.vector-space}.

Let us begin by constructing $f_2$.
Since $X$ is versal, by Lemma~\ref{lem:canHitAnything} 
there exists a
$G$-equivariant $k$-morphism
$\psi \colon S \to X$ such that the image of $\psi_L$ is not contained in $Y$.
Equivalently, there exists a $G$-equivariant rational $k$-map  
$f_2 \colon V \dasharrow X$ such that the image of $(f_2)_L$
is not contained in $Y$. (Note that our construction of $f_2$ makes 
use of the fact that $X$ is $(S, F)$-versal, not just $(T, K)$-versal.)

Let $U$ be a $G$-invariant dense open subvariety of $V$ defined over $k$,
such that $f_2$ is regular on $U$. From now on $f_2$ will denote 
the regular map $U \to X$. 
Note that $(f_2)_L^{-1}(Y)$ is a closed subvariety of $U_L$ and
$(f_2)_L^{-1}(Y) \neq U_L$.
Now recall that by Lemma~\ref{lem.Hilbert90},
$K$-points are dense in $^T V \simeq V_K$ and hence, in $^T U$.
Thus, by Lemma~\ref{lem:arithVersalEquiv}, there exists a $G$-equivariant 
$k$-morphism $f_1 \colon T \to U$ such that
$(f_1)_L(s) \not \in (f_2)_L^{-1}(Y)$.
By our construction, the composition $f_2 f_1$ is a regular $G$-equivariant 
morphism $T \to X$ and $f_2 f_1(s) \not \in Y$, as desired. This completes 
the proof of Theorem~\ref{thm1}(b).
\end{proof}

\section{Proof of Theorem~\ref{thm1}\ref{thm1.c} and \ref{thm1.d}}
\label{sect.proof-of-thm1cd}

Let 
\[ \xymatrix{
S \ar@{->}[r] \ar@{->}[d]  & V_0 \ar@{^{(}->}[r] \ar@{->}[d]  & V \\  
\Spec(F) \ar@{->}[r]^{\eta} & B & } \]
be as in Definition~\ref{def.vector-space}.

\begin{lem} \label{lem:twistByLinRep}
Let $X$, $Y$ be geometrically irreducible $k$-varieties
and suppose $X$ has a $G$-action.
If there exists a dominant rational (resp. birational) 
$F$-map $f \colon Y_F \dasharrow {}^SX$
then there exists a $G$-equivariant dominant 
rational (resp. birational) $k$-map
$V \times Y \dasharrow V \times X$.
\end{lem}

\begin{proof} Choose an open $F$-subvariety $Z \subset Y_F$ such that  
$f_{\, | Z} \colon Z \to { }^S X$ is regular.
By Proposition~\ref{prop:twistAdjunction}, $f$ gives rise to a commutative 
diagram
\[ \xymatrix{
S \times_F Z \ar@{->}[r] \ar@{->}[d]  &
S \times_F X_F \ar@{->}[d] \\
 Z \ar@{->}[r]^(.5){f_{\, | Z}} \ar@{->}[dr]  & ^S X_F \ar@{->}[d] \\
 & \Spec(F) \, ,} \]
where the vertical maps in the square are $G$-torsors.
By a well known property of torsors, since the map $Y_F \dasharrow {}^SX$ is
dominant (resp.~birational), so is the top horizontal map.

Note that $Z$ may not 
descend to a variety over $k$; however, after replacing $B$ by 
a dense open subvariety, we may assume that the open immersion $Z \subset Y_F$ 
descends to $B$, i.e., there exists a $k$-variety $Z'$ such that the pull-back
diagram 
\[ \xymatrix{ & Y_F \ar@{->}[ddl] \ar@{->}[r]
& Y_B \ar@{->}[ddl] \\
Z \ar@{->}[r] \ar@{->}[d] \ar@{^{(}->}[ur] 
 & Z' \ar@{->}[d] \ar@{^{(}->}[ur] &  \\  
\Spec(F) \ar@{->}[r]^{\eta}   & B & } \] 
commutes and $Z' \subset Y_B$ is an open immersion; 
see~\cite[Proposition 2.7.1(x)]{EGA4}.

By the naturality of the fiber product operation, 
the $G$-equivariant $F$-map
\[ S \times_F Z \to S \times_F X_F \]
is $G$-equivariantly isomorphic to an $F$-map
\[ (V_0 \times_B Z')_F \to (V_0 \times_B X_B)_F \ ; \]
see~\cite[Corollaire 3.3.10]{EGA1}.
Shrinking $B$ once more, we obtain a dominant
(resp. birational) $G$-equivariant map
$V_0 \times_B Z' \to V_0 \times_B X_B$ 
such that the commutative triangle on the left is the pull-back of 
the commutative triangle on the right to the generic point $\eta$:
\[ \xymatrix{
S \times_F Z \ar@{->}[r] \ar@{->}[dr]  & S \times_F X_F \ar@{->}[d] & 
  &  V_0 \times_B Z' \ar@{->}[r] \ar@{->}[dr]  & V_0 \times_B X_B \ar@{->}[d] \\ 
 & \Spec(F)  \ar@{->}[rrr]^{\eta} & & &  B \, ,} . \]

Since $Z' \hookrightarrow Y_B$ is an open immersion, 
we obtain a dominant (resp. birational)
$G$-equivariant rational map
\begin{equation} \label{e.lem4.7}
V_0 \times_B Y_B \dasharrow V_0 \times_B X_B
\end{equation}
We now note that $V_0 \times_B Y_B \simeq V_0 \times_k Y$ 
and $V_0 \times_B \, X_B \simeq  V_0 \times_k X$, where $\simeq$ denotes
$G$-equivariant isomorphism over $k$. Thus \eqref{e.lem4.7}
gives us a dominant (resp. birational) $G$-equivariant rational $k$-map 
$V_0 \times_k Y \dasharrow V_0 \times_k X$ or equivalently, 
a dominant (resp. birational) $G$-equivariant rational $k$-map 
$V \times_k Y \dasharrow V \times_k X$, as desired.
\end{proof}

\begin{proof}[Proof of Theorem~\ref{thm1}\ref{thm1.c}]

$\implies$:
Suppose $X$ is very versal, i.e., there exists a dominant
$G$-equivariant map $f \colon W \dasharrow X$. Then for any twisting pair 
$(T, K)$, 
the $K$-rational map $\, ^T f \colon \, ^T W \dasharrow \, ^T X$ 
is dominant; see Corollary~\ref{cor:twistsPreserve}.
By Lemma~\ref{lem.Hilbert90}, $^T W \simeq_K W_K$. Thus $^T X$ 
is $K$-unirational.

$\impliedby$:
By assumption, there exists a dominant rational map
$\bbA_F^n \dasharrow {}^SX$ for some
integer $n$. By Lemma~\ref{lem:twistByLinRep}, we obtain a dominant rational
$G$-equivariant $k$-map $V \times \bbA_k^n \dasharrow V \times X$, where
$G$ acts trivially on $\bbA_k^n$. Composing 
this map with the projection $V \times X \to X$, we see that $X$ 
is very versal.
\end{proof}

\begin{proof}[Proof of Theorem~\ref{thm1}\ref{thm1.d}]

$\implies$: Suppose the $G$-action on $X$ is stably birationally linear, i.e.,
there exists a $G$-equivariant birational isomorphism 
$\phi \colon X \times W_1 \biratarrow W_0$ for some linear 
representations $G \to \GL(W_1)$ and $G \to \GL(W_0)$ defined over $k$.
Twisting $\phi$ by a twisting pair $(T, K)$, we obtain
a birational isomorphism 
\[ ^T \phi \colon (^T X) \times_K (^T W_1) \biratarrow {}^T W_0 \, . \] 
Since $^T W_1$ and $^T W_0$ are affine spaces over $K$ 
(cf.~Lemma~\ref{lem.Hilbert90}), this tells us that
$^T X$ is stably rational over $K$.

$\impliedby$:
By assumption,
there is a birational isomorphism
$\bbA_F^n \biratarrow {}^SX \times \bbA_F^m$ 
defined over $F$. Now note that
${}^SX \times \bbA_F^m \simeq {}^S(X \times \bbA_k^m)$, where 
$G$ acts trivially on $\bbA_k^m$; cf.~Corollary~\ref{cor:twistsPreserve}(a).
By Lemma~\ref{lem:twistByLinRep}, we obtain a
$G$-equivariant birational isomorphism 
$V \times \bbA_k^n \biratarrow V \times X \times \bbA_k^m$,
defined over $k$.  Here $G$ acts trivially on both $\bbA_k^n$ and $\bbA_k^m$.
This shows that the $G$-action on $X$ is stably birationally linear.
\end{proof}

\section{Forms of Moduli Spaces}
\label{sec:ModuliSpaces}

Throughout this section, the base field $k$ will be assumed to be of
characteristic $0$ and $\moduli{g}{n}$ will denote the moduli space 
of stable curves of genus $g$
with $n$ marked points viewed as a $k$-variety. The symmetric group 
$\Sym_n$ acts on $\moduli{g}{n}$ by permuting the $n$ marked points. 
If $k = \bbC$, A.~Massarenti~\cite{massarenti} 
showed that $\Sym_n$ is, in fact, the full automorphism group
of $\moduli{g}{n}$ whenever $2g + n \geqslant 5$.  This extends
earlier work of A.~Bruno and M.~Mella~\cite{bm} in 
the case where $g = 0$. In other words, if $2g + n \geqslant 5$ then
the natural homomorphism   
\[ \iota  \colon \Sym_n \to \Aut(\moduli{g}{n}) \]
is an isomorphism.  Here we view $\Aut(\moduli{g}{n})$ as a group scheme 
over $k$. (The automorphisms of a complete variety 
carries a natural structure of an algebraic group scheme; cf.~\cite{mo}.  
A priori this group scheme is neither affine nor of finite type over $k$.)
By the Lefschetz principle $\iota$ is an isomorphism for any 
algebraically closed base field $k$ of characteristic zero.
In fact, the same is true if $k$ is not assumed to be algebraically 
closed; this can be seen by descent~\cite[Proposition 2.7.1]{EGA4},
after passing to $\overline{k}$.

\begin{thm} \label{thm.moduli} Suppose $K/k$ is a field extension.
Then 
\begin{enumerate}
\item For $n \geqslant 5$,
every $K$-form of $\moduli{0}{n}$ is $K$-unirational.
\item For $3 \leqslant n \leqslant 9$,
every $K$-form of $\moduli{1}{n}$ is $K$-unirational.
\end{enumerate}
\end{thm}

\begin{proof}
Let $g = 0$ in part (a) and $g = 1$ in part (b).
Then $2g + n \geqslant 5$ and hence,
$\Sym_n$ is the full automorphism 
group of $\moduli{g}{n}$, in both parts.
Consequently, every $K$-form  
of $\moduli{g}{n}$, over a field extension $K/k$
is isomorphic to $^T X$ for some $\Sym_n$-torsor $T \to \Spec(K)$.
By Theorem~\ref{thm1}\ref{thm1.c} it suffices to show that
the $\Sym_n$-action on $\moduli{g}{n}$ 
is very versal.

(a) Consider the following composition of dominant $\Sym_n$-equivariant 
rational maps: 
\[ (\bbA^2)^n \dasharrow (\mathbb P^1)^n \dasharrow \moduli{0}{n} \, .\]
Here the first map is $n$ parallel applications of
the natural projection $\bbA^2 \setminus \{ 0 \} \to \bbP^1$, and the
second map takes an $n$-tuple of distinct points on $\bbP^1$
to its class in $\moduli{0}{n}$.
The symmetric group $\Sym_n$ acts on the $2n$-dimensional affine 
space $(\bbA^2)^n$ linearly, by permuting the $n$ factors of $\bbA^2$.
This shows that the action of $\Sym_n$ on $\moduli{g}{n}$ is very versal.

\smallskip
(b) 
It suffices to consider the case where $n = 9$. Indeed, 
the natural map $\moduli{g}{n+1} \dasharrow \moduli{g}{n}$ 
which ``forgets" the last marked point,
is $\Sym_n$-equivariant and dominant for any $n \geqslant 0$. Thus if the
$\Sym_{n+1}$-action on $\moduli{g}{n+1}$ is very versal then so is the
$\Sym_n$-action on $\moduli{g}{n}$. 

To show that 
the $\Sym_{9}$-action on $\moduli{1}{9}$ is very versal, we consider 
the following composition of dominant $\Sym_9$-equivariant maps:
\[ (\bbA^3)^9 \dasharrow (\bbP^2)^9 \dasharrow \moduli{1}{9} \, . \]
Here the first map projects each 
$\bbA^3 \setminus \{ 0 \}$ to $\bbP^2$. The second map takes a $9$-tuple 
of points $p_1, \dots, p_9 \in \mathbb P^2$ to the point
of $\moduli{1}{9}$ represented by $(C, p_1, \dots, p_9)$, where
$C$ is the plane cubic curve passing through $p_1, \dots, p_9$. Note that
$C$ is uniquely determined and non-singular
for $p_1, \dots, p_9$ in general position.
\end{proof}

\begin{remark}
It is well known that $\moduli{0}{5}$ 
is a Del Pezzo surface of degree $5$ and that every
Del Pezzo surface of degree $5$ over a field $K/k$ is 
a $K$-form of $\moduli{0}{5}$.  Thus, for $n = 5$, 
we recover the theorem of H.~P.~F.~Swinnerton-Dyer~\cite{sd} 
about the existence of rational points 
on such surfaces (in characteristic $0$ only). 
For alternative proofs of Swinnerton-Dyer's
theorem, see \cite{sb} and~\cite{skoroESD}.
\end{remark}

\begin{remark} \label{rem.genus1}
Over $\bbC$, $\moduli{1}{n}$ itself is known to be rational 
if $n \leqslant 10$ (see~\cite{belorousski} or~\cite[Section 2]{cf}) 
and non-unirational for $n \geqslant 11$; see~\cite[p.~583]{cf}.

Note also that by~\cite[p.~2]{massarenti}, $\Aut(\moduli{1}{2}) \simeq 
\bbG_m^2$, is a $2$-dimensional torus.  By  Hilbert's Theorem 90,
$H^1(K, \bbG_m^2) = \{ 1 \}$, i.e., every torsor $T \to \Spec(K)$ is split
for every field extension $K/\bbC$.
Thus $\moduli{1}{2}$ has no non-trivial twisted forms:
${ }^T \moduli{1}{2} \simeq (\moduli{1}{2})_K$.  
Since $\moduli{1}{2}$ is rational over $\bbC$
(see~\cite[Proposition 2.1]{massarenti} or the references above), 
we conclude that every $K$-form of $\moduli{1}{2}$ is rational 
over $K$, for every field extension $K/\bbC$. 
\end{remark}

\begin{remark}
The conclusion of Theorem~\ref{thm.moduli} fails for
$\moduli{0}{4} \simeq \moduli{1}{1} \simeq \bbP^1$ (over $\bar{k}$).
Indeed, in general, a form of $\bbP^1$
(a $1$-dimensional Brauer-Severi variety) 
has no rational points.
\end{remark}

\begin{remark} If $g \geqslant 23$ then
$\moduli{g}{0}$ is not unirational,
and hence, neither is $\moduli{g}{n}$  
for any $n \geqslant 1$. For each $g \leqslant 22$, A. Logan~\cite{logan}
showed that there is an explicit function $f(g)$ such that
$\moduli{g}{n}$ is not unirational for any $n \geqslant f(g)$. Some
(but not all) of the spaces $\moduli{g}{n}$ with $n < f(g)$ 
have since been shown to be unirational;
for details and further references, see~\cite[Introduction]{bcf}.
Beyond the results of this section, we do not know which 
of the finitely many unirational 
spaces $\moduli{g}{n}$ have the property that
every form of $\moduli{g}{n}$ over every $K/k$ is also unirational.
In particular, we do not know if Theorem~\ref{thm.moduli} remains valid for
$\moduli{1}{10}$.
\end{remark}

\section{Homogeneous Spaces}
\label{sect.homogSpace}

\begin{prop} \label{prop:homogSpace}
Let $A$ be an algebraic group over $k$.
If $\Char(k) > 0$, assume that $A$ is reductive.
Suppose $G$ and $B$ are closed subgroups 
of $A$, and $X := A/B$ is 
geometrically irreducible.  Consider $X$ as 
a $G$-variety.  The following are equivalent:
\begin{enumerate}
\item $X$ is very versal,
\item $X$ is versal,
\item $X$ is weakly versal,
\item the image of the natural map $H^1(K, G) \to H^1(K, A)$ 
is contained in the image of the natural map $H^1(K, B) \to H^1(K, A)$ 
for every field extension $K/k$ where $K$ is infinite.
\end{enumerate}
\end{prop}

\begin{proof}
Let $(T, K)$ be a twisting pair.
In view of Proposition~\ref{prop:TGaction}, $^T X$ is 
a homogeneous space for the twisted group $^T A$. 

By Theorem~\ref{thm1} it suffices to show that the following conditions 
are equivalent:
\begin{enumerate}[label=(\roman*)]
\item $^T X$ has a $K$-point, \label{homog_a}
\item there exists a dominant $^T A$-equivariant 
map $f \colon  \, ^T A \to \, ^T X$ defined over $K$,\label{homog_b}
\item $^T X$ is $K$-unirational, \label{homog_c}
\item $K$-points are dense in $^T X$, \label{homog_d}
\item the class of $T$ lies in the image of the natural 
map $H^1(K, B) \to H^1(K, A)$. \label{homog_e}
\end{enumerate}

\smallskip
\ref{homog_a} $\Longrightarrow$ \ref{homog_b}:
By Proposition~\ref{prop:TGaction} 
$\, ^T A$ acts on $^T X$. If $p \in X(K)$, we can 
define $f$ to be the orbit map $f(g) = g \cdot p$.
Passing to a splitting field of $T$, we see that $f$ is dominant.

\ref{homog_b} $\Longrightarrow$ \ref{homog_c}:
Let $({}^TA)^{\circ}$ be the connected component of the identity in ${}^TA$.
Note that
the composition $({}^TA)^{\circ} \to {}^TA \to {}^TX$ is also dominant since
$^TX$ is geometrically irreducible.

We use a theorem of Chevalley~\cite[Theorem 18.2(ii)]{borel} which
asserts that a connected linear algebraic group is unirational
when either the group is reductive or the ground field is perfect.
Thus, under our assumptions on $A$,
$({}^TA)^{\circ}$ is unirational.

The implications \ref{homog_c} $\Longrightarrow$ \ref{homog_d} 
$\Longrightarrow$ \ref{homog_a} are obvious. 

\ref{homog_b} $\Longleftrightarrow$ \ref{homog_e} is proved in
\cite[Proposition I.5.37]{serre-gc}; see 
also~\cite[Proposition 1.11]{Springer}.
\end{proof}

\begin{remark} \label{rem.perfect} Note that the above argument fails if 
$\Char(k) > 0$ and we do not assume that $A$ is reductive, even if
$k$ is a perfect field. Indeed, in this case Chevalley's theorem tells us that
the group $A^0$ is unirational over $k$. However, the twisted group 
$({}^TA)^{\circ}$, defined over an (a priori arbitrary) extension $K/k$
may not be unirational over $K$, since $K$ may not be perfect.
\end{remark}

\begin{example} \label{ex.special} 
We record several interesting special cases of 
Proposition~\ref{prop:homogSpace} when $A$ is connected.

(a) Suppose $B = \{ 1 \}$.  Then the translation 
action of a subgroup $G$ on $A$ is versal if and only 
if the natural map $H^1(K, G) \to H^1(K, A)$ is trivial 
for every field extension $K/k$, where $K$ is infinite.
The same is true whenever $B$ is a special group, 
i.e., whenever $H^1(K, B)$ is trivial for every $K/k$. 

(b) Setting $B = G$ 
yields the following: For any closed subgroup $G \subset A$, 
the translation action of $G$ on $A/G$ is very versal. 

(c) If $B$ is the normalizer of a maximal torus in $A$,
we see that the translation action of $G$ on $A/B$ is 
very versal for any $G \subset A$. 
This is because the natural map $H^1(K, B) \to H^1(K, A)$ 
is surjective for every field extension $K/k$; 
see~\cite[III.4.3, Lemma 6]{serre-gc}
if $K$ is perfect and~\cite[Corollary 5.3]{cgr2} otherwise.
\end{example}

We may also apply Chevalley's theorem in cases where $G$ acts via
group automorphisms rather than by translations.

\begin{example} \label{ex.groups}
{\rm (}cf.~\cite[Proposition 3.3]{ckpr}{\rm )}
Let $H$ be a connected algebraic group. If
$\Char(k) > 0$, assume that $H$ is
reductive. Then every action of an algebraic group $G$ on
$H$ by group automorphisms is very versal.
\end{example}

\begin{proof}
Let $(T, K)$ be a twisting pair. By Proposition~\ref{prop:twistedGroup}, 
$^T H$ carries the structure of an affine algebraic group over $K$. By 
Chevalley's theorem~\cite[Theorem 18.2(ii)]{borel}
$^T H$ is unirational over $K$. The desired conclusion 
now follows from Theorem~\ref{thm1}\ref{thm1.c}.
\end{proof}

\section{$p$-versality}
\label{sect.p-versality}

Throughout this section, $p$ is a prime number.

\begin{defn}
Let $G/k$ be an algebraic group and let $X/k$ be
an irreducible $G$-variety.  
We say that $X$ is
\begin{itemize}
\item \emph{weakly $p$-versal} if for every twisting pair $(T, K)$, 
there exists a field extension $L/K$ of degree prime to $p$
and a $G$-equivariant $k$-morphism $T_L \to X$,

\item \emph{$p$-versal} if every $G$-invariant
dense open subvariety $U \subset X$ is weakly $p$-versal
(cf.~\cite[Section 2.2]{mED}).
\end{itemize}
\end{defn}

Recall that a field $L$ is called \emph{$p$-closed} if the degree 
of every finite field extension of $L$ is a power of $p$.
For every field $K$, there exists an algebraic extension $K^{(p)}/K$,
such that $K^{(p)}$ is $p$-closed and, for every finite subextension 
$K \subset K' \subset K^{(p)}$, the degree $[K' : K]$ is prime to $p$.
The field $K^{(p)}$ satisfying these conditions is unique up to
$K$-isomorphism; it usually called the $p$-{\em closure} 
of $K$ and is denoted by $K^{(p)}$. For details, 
see~\cite[Proposition 101.16]{EKM}.

\begin{lem} \label{lem:ptwistWeakVersal}
Let $X$ be a geometrically irreducible $G$-variety defined over $k$.
Then the following conditions are equivalent:

\begin{enumerate}
\item $X$ is weakly $p$-versal, \label{lem:ptwist(a)}
\item for every twisting pair $(T, K)$, $^T X$ has a point
whose degree over $K$ is prime to $p$,
\label{lem:ptwist(b)}
\item  for every twisting pair $(T, K)$, 
$^T X$ has a $0$-cycle whose degree is prime to $p$,
\label{lem:ptwist(c)}
\item  for every twisting pair $(T, K)$, the variety
$(^T X)_{K^{(p)}}$ has a $0$-cycle of degree $1$,
\label{lem:ptwist(d)}
\item  for every twisting pair $(T, K)$, the variety
$^T X$ has a $K^{(p)}$-point.
\label{lem:ptwist(e)}
\end{enumerate}
\end{lem}

\begin{proof} \ref{lem:ptwist(a)} $\Longleftrightarrow$ \ref{lem:ptwist(b)}:
By Proposition~\ref{prop:twistAdjunction},
the existence of an $L$-point of $^T X$ is equivalent to the existence
of a
$G$-equivariant $k$-morphism $T_L \to X$.

\ref{lem:ptwist(b)} $\implies$ \ref{lem:ptwist(c)} is obvious.

\ref{lem:ptwist(c)} $\implies$ \ref{lem:ptwist(d)}: Suppose
$Z \subset {}^T X$ is a $0$-cycle of degree $d$, where $d$ is prime to $p$.
Since the degree of every closed point of 
$(^T X)_{K^{(p)}}$ is a power of $p$, there exists a $0$-cycle
$Z' \subset (^T X)_{K^{(p)}}$ whose degree is a power of $p$.
A desired $0$-cycle of degree $1$ on $(^T X)_{K^{(p)}}$ can 
then be constructed as an integer linear combination of $Z$ and $Z'$.

\ref{lem:ptwist(d)} $\implies$ \ref{lem:ptwist(e)}: 
(cf. \cite[Example 13.1]{fulton}) This is immediate 
from the fact that the degree of every closed point on 
$(^T X)_{K^{(p)}}$ is a power of $p$.

\ref{lem:ptwist(e)} $\implies$ \ref{lem:ptwist(b)}: Every 
$K^{(p)}$-point of $^T X$ descends to a finitely generated 
subextension $K \subset L \subset K^{(p)}$. The field $L$ is
then a finite extension of $K$ whose degree is prime to $p$.
\end{proof} 

\begin{thm} \label{thm.p-versal} Let $G$ be an algebraic group acting on
a smooth geometrically irreducible $k$-variety $X$. 
Then $X$ is $p$-versal if and only if $X$ is weakly $p$-versal.
\end{thm}

\begin{proof} We will assume that $X$ is weakly $p$-versal and prove that
$X$ is $p$-versal; the other direction is obvious.

Let $U \subset X$ be a $G$-invariant dense open subvariety.  
We want to show that
$U$ is weakly versal. Let $(T, K)$ be a twisting pair. By 
Lemma~\ref{lem:ptwistWeakVersal}, it suffices to prove 
that if $^T X$ has a $0$-cycle 
whose degree is prime to $p$, then so does $^T U$. Since
$^T U$ is a dense open subvariety of $^T X$
(see Corollary~\ref{cor:twistsPreserve}(b)),
this is a special case of Chow's Moving Lemma~\cite[Section 11.4]{fulton}.
\end{proof}

\begin{remark} The above argument only requires  
Chow's Moving Lemma for $0$-cycles. This special case of
the Moving Lemma is quite elementary: to prove it, one
constructs a smooth curve passing through
a given $0$-cycle, then moves the $0$-cycle along this curve.
\end{remark}

\begin{remark} \label{remark.moving-lemma}
We thank an anonymous referee for pointing out that
Theorem~\ref{thm.p-versal} is closely related 
to the fact that every $p$-closed field $L$ is
``ample'' (or ``large''). That is, if an irreducible $L$-variety
has a smooth $L$-point, then the $L$-points are Zariski dense.
This property of $p$-closed fields was noticed by 
J.-L.~Colliot-Th{\'e}l{\`e}ne~\cite{CTample} in the case where
$L$ is perfect. In the case where $L$ is not perfect it 
was proved by M.~Jarden~\cite{jarden}.
\end{remark}

\begin{cor} \label{cor:GversalIsGpversal}
Let $X/k$ be a geometrically irreducible generically smooth
$G$-variety.
\begin{enumerate}
\item Assume that $G$ has a closed subgroup $H$ whose index is 
finite and prime to $p$. Then the $G$-action on $X$ is $p$-versal 
if and only if the restricted $H$-action is $p$-versal.

\item Suppose there exists a smooth $k$-point $x \in X(k)$ 
such that the orbit $G \cdot x$ is finite and $\deg([G \cdot x])$
is prime to $p$. Then the $G$-action on $X$ is $p$-versal.
\end{enumerate}
\end{cor}

\begin{proof} After replacing $X$ by its smooth locus, 
we may assume that $X$ is smooth.  

(a) From the proof of~\cite[Lemma 4.1]{mr1},
for any field $K/k$, the map $H^1(K,H) \to H^1(K,G)$ is $p$-surjective.
That is, for any $\alpha \in H^1(K,G)$ there exists a finite 
extension $L/K$ of degree prime to $p$ such that $\alpha_L$ 
lies in the image of the natural map $H^1(L,H) \to H^1(L,G)$.  
If $K$ is $p$-closed, then $[L: K]$ is a power of $p$, so
$L = K$, and the map $H^1(K,H) \to H^1(K,G)$ is
surjective.  In other words, for any $H$-torsor $T \to \Spec(K)$, 
there exists a $G$-torsor $T' \to \Spec(K)$ such that $^TX$ 
and $^{T'}X$ become isomorphic over $K^{(p)}$. In particular,
$^T X$ has a $K^{(p)}$-point if and only if $^{T'} X$ has a $K^{(p)}$-point.
Lemma~\ref{lem:ptwistWeakVersal}
now tells us that the $G$-action on $X$ is weakly $p$-versal if and only
the $H$-action is weakly $p$-versal.
By Theorem~\ref{thm.p-versal}, the same 
is true if ``weakly $p$-versal'' is replaced by ``$p$-versal''. 

(b) Let $H$ be the stabilizer of $x$ in $G$. Then $x$ is 
fixed by $H$, and the
index $[G:H] = \deg([G \cdot x])$ is finite and prime to $p$. 
By Proposition~\ref{prop.weakly-versal},
the $H$-action on $X$ is weakly versal. 
By Theorem~\ref{thm.p-versal}, the $H$-action on $X$ is $p$-versal.
By part (a) the $G$-action on $X$ is $p$-versal as well.
\end{proof}

We also note the following immediate consequence
of Theorem~\ref{thm.p-versal} and Lemma~\ref{lem:ptwistWeakVersal},  
in the spirit of Theorem~\ref{thm1}.

\begin{cor} \label{cor.zero-cycle}
A $G$-action on a smooth geometrically irreducible variety $X$ is
$p$-versal for every prime $p$ if and only if,
for every twisting pair $(T, K)$, ${}^T X$ 
has a $0$-cycle of degree $1$.
\qed
\end{cor}

Every versal $G$-variety is clearly $p$-versal for every prime $p$. 
However, the converse is not true in general, even if $G= \{ 1 \}$;
after all, there exist $k$-varieties with $0$-cycles
of degree $1$ but no $k$-points. 
On the other hand, no counterexample is known for the following 
weaker statement:

\begin{conj}[cf.~\cite{duncan2}]
\label{conj:GpVersalEnough}
Let $G$ be a finite constant group, $X$ be a $G$-variety and
$G_p$ be a Sylow $p$-subgroup of $G$.
If $X$ is $G_p$-versal for every prime $p$, then $X$ is $G$-versal.
\end{conj}

Note that the key assumption here is that $X$ is versal and not just
$p$-versal as a $G_p$-variety.

\begin{remark}
It is natural to define a $G$-variety $X$ to be ``very $p$-versal'' 
if there exists a linear representation $V$, and a diagram of
dominant rational $G$-equivariant maps of the form
\[ \xymatrix{ V' \ar@{-->}[d] \ar@{-->}[dr] &   \cr
V & X \, , } \]
where the degree of $V' \dasharrow V$ is prime to $p$. 
(Note that $V'$ is not required to be a vector space.)
Under mild assumptions on $X$ this notion also 
turns out to be equivalent to $p$-versality. 
\end{remark} 

\section{Projective representations}
\label{sect.proj-rep}

Let $G$ be a finite subgroup of $\PGL_n$ defined over $k$ and
$G'$ be the preimage of $G$ in 
$\GL_n$. The diagram
\begin{equation} \label{e.group-extension}
\xymatrix{ 1 \ar@{->}[r] & \bbG_m \ar@{->}[r] & 
G'  
\ar@{->}[r] &  G \ar@{->}[r] 
& 1 
} 
\end{equation}
gives rise to the connecting morphism  
$\partial_K: H^1(K, G) \to  H^2(K, \mathbb G_m)$ 
for every field $K/k$.

\begin{prop} \label{prop:PnVersal} 
{\rm (}cf.~\cite[Corollary 3.4]{duncan2}{\rm )}
Let $G$ be a finite subgroup of $\PGL_n$ defined over $k$.
Then the following conditions are equivalent:

\begin{enumerate}
\item The $G$-action on $\bbP^{n-1}$ is stably birationally linear,
\label{Pn1}
\item the $G$-action on $\bbP^{n-1}$ is very versal, \label{Pn2}
\item the $G$-action on $\bbP^{n-1}$ is versal, \label{Pn3}
\item the $G$-action on $\bbP^{n-1}$ is weakly versal, \label{Pn4}
\item the $G$-action on $\bbP^{n-1}$ is $p$-versal for every prime $p$, 
\label{Pn5}
\item
$\partial_K = 0$ for every $K/k$,
\label{Pn6}
\item $G$ lifts to a subgroup of $\GL_n$, i.e., the exact sequence
\eqref{e.group-extension} splits.
\label{Pn7}
\end{enumerate}
\end{prop}

\begin{proof} 
Let $(T, K)$ be  a $G$-twisting pair.
Then $X = \,  ^T (\bbP^{n-1})$ is a Brauer-Severi variety over $K$ 
whose class is $\partial_K([T])$, where $[T]$ is the class of $T$
in $H^1(K, G)$.
It is well known that a Brauer-Severi variety $X$ over $K$ is
$K$-rational if and only if $X$ has a $0$-cycle of degree $1$ if
and only if the class of $X$ in $H^2(K, \mathbb G_m)$ is trivial.
This shows that conditions \ref{Pn1} - \ref{Pn6} are all equivalent
by Theorem~\ref{thm1} and Corollary~\ref{cor.zero-cycle}.

\ref{Pn7} $\Longrightarrow$ \ref{Pn2}. If $G$ lifts to $\GL_n$ then
the natural projection map 
$\mathbb{A}^n \setminus \{ 0 \} \to \mathbb{P}^{n-1}$ is dominant
and $G$-equivariant, and 
\ref{Pn2} follows.

\ref{Pn6} $\Longrightarrow$ \ref{Pn7}. 
By~\cite[Theorem 4.4 and Remark 4.5]{km}, \ref{Pn6} implies 
\[ \gcd_{\rho} \, \dim(\rho) = 1 \, , \]
as $\rho$ ranges over representations $G' \to \GL(V)$ such that
$\rho(t) = t \operatorname{I}_V$ for every $t \in \mathbb G_m$.
Here $I_V$ is the identity map on $V$.
Thus there exist representations $\rho_1, \dots, \rho_m$ of $G$ and
integers $d_1, \dots, d_m$ such that the multiplicative character 
$\chi = \det(\rho_1)^{d_1} \dots \det(\rho_m)^{d_m} \colon G' \to \bbG_m$  
has the property that $\chi(t) = t$
and hence splits the sequence \eqref{e.group-extension}
($\operatorname{Ker}(\chi)$ is a complement of $\mathbb G_m$ in $G'$).
\end{proof}

\section{Group actions on quadric and cubic hypersurfaces}
\label{sect.quadric-cubic}

\begin{lem} \label{lem:hilbertPoly}
Let $V$ be a finite-dimensional $k$-vector space, $G \to \GL(V)$
be a linear representation and $X$ be a closed $G$-invariant subvariety
of $\bbP(V)$. 

\begin{enumerate}
\item For any twisting pair $(T, K)$, 
$^T \bbP(V)$ is $K$-isomorphic to $\bbP(V)_K$.
\label{lem:hilbertPoly_a}

\item
The inclusion $\iota \colon X \hookrightarrow \bbP(V)$ induces a
closed embedding $^T \iota \colon {}^TX \hookrightarrow \, \bbP(V)_K$ 
with the same Hilbert polynomial as $X$.
\label{lem:hilbertPoly_b}

\item
If $d := \deg(X)$ is prime to $p$ then $X$ is $p$-versal.
\label{lem:hilbertPoly_c}
\end{enumerate}
\end{lem}

\begin{proof} 
\ref{lem:hilbertPoly_a} By Lemma~\ref{lem.Hilbert90}, $^TV \simeq V_K$. 
The $(T, K)$-twist of the natural projection $V \dasharrow
\mathbb P(V)$, is thus
a dominant rational map $V_K \dasharrow {}^T\bbP(V)$.  Consequently,
the Brauer-Severi variety ${}^T\bbP(V)$ has a $K$-point, and
part \ref{lem:hilbertPoly_a} follows.

\ref{lem:hilbertPoly_b} Since the embeddings $^T\iota : {}^TX \to \bbP(V)_K$
and $\iota : X \to \bbP(V)$ become projectively equivalent over an algebraic
closure $\bar{K}$, they have the same Hilbert polynomial.

\ref{lem:hilbertPoly_c} By part~\ref{lem:hilbertPoly_b}, 
$^T X$ is a closed subvariety of $\mathbb P(V_K)$ of degree $d$,
defined over $K$.  Intersecting $^T X$ with a suitable
linear subvariety of $\bbP(V_K)$ of complementary dimension,
we obtain a smooth $0$-cycle of degree
$d$ on $^T X$ defined over $K$.  Since $d$ is not divisible by $p$, 
Lemma~\ref{lem:ptwistWeakVersal} and Theorem~\ref{thm.p-versal} now
tell us that $X$ is $p$-versal.  
\end{proof}

\begin{thm} \label{thm:quadraticVersal}
Let $G$ be an algebraic group over $k$, 
$G \to \GL(V)$ be a finite-dimensional linear representation over $k$,
and $X \subset \bbP(V)$ be a smooth $G$-invariant quadratic
hypersurface.  Assume $\dim(V) \geqslant 3$.  The following are equivalent:
\begin{enumerate}
\item $X$ is stably birationally linear, \label{quad_bl}
\item $X$ is very versal, \label{quad_vv}
\item $X$ is versal, \label{quad_versal}
\item $X$ is weakly versal, \label{quad_weak}
\item $X$ is $2$-versal. \label{quad_2versal}
\end{enumerate}

\noindent
Assume further that $G$ is finite, and $G_2$ is a Sylow $2$-subgroup of
$G$. Then conditions
\ref{quad_bl}  -- \ref{quad_2versal} are equivalent to
\begin{enumerate}[resume]
\item $X$ is versal for the action of $G_2$. \label{quad_Sylow}
\end{enumerate}
\end{thm}

\begin{proof}
Let $(T, K)$ be a twisting pair.
By Lemma~\ref{lem:hilbertPoly}, $Q := {}^T X$
is an irreducible quadratic hypersurface in $\bbP^n_K$ defined over $K$. The
equivalence of conditions \ref{quad_bl}--\ref{quad_weak} now follows from
Theorem~\ref{thm1} and the
following well-known property of irreducible quadric hypersurfaces
$Q \subset \mathbb P(V)_K$:
\begin{equation} \label{e.quadric}
\text{if $Q$ has a smooth $K$-point then $X$ is $K$-rational.}
\end{equation}
The equivalence of \ref{quad_bl} and
\ref{quad_2versal} is an immediate consequence of Springer's theorem:
if $Q$ has a smooth $L$-point for some odd degree extension $L/K$ then $Q$ 
has a smooth $K$-point.

If $G$ is a finite group then \ref{quad_Sylow} $\implies$ \ref{quad_2versal}
by Corollary~\ref{cor:GversalIsGpversal}(a). On the other 
hand,~\ref{quad_vv} $\implies$ $X$ is very versal 
as a $G_2$-variety $\implies$ \ref{quad_Sylow}.
\end{proof}

If we replace a quadric hypersurface by a cubic hypersurface
of dimension $\geqslant 2$ then property~\eqref{e.quadric} in the 
proof of Theorem~\ref{thm:quadraticVersal}
remains true, provided that ``rational'' is replaced by 
``unirational'', and Springer's Theorem
becomes an open conjecture.  The precise statements are as follows.

\begin{thm} \label{thm:cubicUnirational}
{\rm (}\cite{kollar}{\rm )}
Let $X \subset \bbP^n_k$ be a smooth cubic hypersurface where $n \geqslant 3$.
If $X$ has a $k$-point then $X$ is $k$-unirational. 
\end{thm}

\begin{conj}[J.~W.~S.~Cassels, P.~Swinnerton-Dyer; see~\cite{coray1}] 
\label{conj:CSD}
Suppose $X \subset \bbP^n_k$ is a cubic hypersurface. 
If $X$ has a $0$-cycle of degree prime to $3$, 
then $X$ has a $k$-point.
\end{conj}

The argument we used to prove Theorem~\ref{thm:quadraticVersal} now yields 
the following analogous statement for cubic hypersurfaces.

\begin{thm} \label{thm:cubicVersal}
Let $G$ be an algebraic group over $k$,
$G \to \GL(V)$ be a finite-dimensional linear representation over $k$
and $X \subset \bbP(V)$ be a smooth $G$-invariant cubic hypersurface.
Assume $\dim(V) \geqslant 4$.  
The following are equivalent:
\begin{enumerate}
\item $X$ is very versal, \label{cubic_vv}
\item $X$ is versal, \label{cubic_versal}
\item $X$ is weakly versal. \label{cubic_weak}
\end{enumerate}
Now suppose $G$ is finite, $G_3$ is a Sylow $3$-subgroup of $G$, and
Conjecture~\ref{conj:CSD} holds. Then
\ref{cubic_vv} 
$\Leftrightarrow$
\ref{cubic_versal}
$\Leftrightarrow$
\ref{cubic_weak} 
$\Leftrightarrow$
\ref{cubic_3versal}
$\Leftrightarrow$
\ref{cubic_Sylow}, where
\begin{enumerate}[resume]
\item $X$ is $3$-versal, \label{cubic_3versal}
\item $X$ is versal for the action of $G_3$. \label{cubic_Sylow}
\qed
\end{enumerate}
\end{thm}

Note that Conjecture~\ref{conj:CSD} is known to be true only for $n = 2$;
see,~e.g.,~\cite{coray1}.  In the statement of Theorem~\ref{thm:cubicVersal} 
we are assuming that $n = \dim(V) - 1 \geqslant 3$.

\begin{cor} \label{cor.fixed-pt}
Suppose an algebraic group $G$ acts on a smooth cubic
hypersurface $X$ as in Theorem~\ref{thm:cubicVersal}. 
If $G$ fixes a $k$-point $x \in X(k)$ then $X$ is $G$-versal.
\end{cor} 

\begin{proof}
By Proposition~\ref{prop.weakly-versal}
the $G$-action on $X$ is weakly versal. 
Theorem~\ref{thm:cubicVersal} now tells us that this action is
versal.
\end{proof}

We now recall the definitions of two important numerical invariants of 
a finite group $G$. The {\em essential dimension}, $\ed(G)$ of $G$ 
is the minimal
dimension of a versal $G$-variety with a faithful $G$-action; 
see~\cite{br}, \cite{icm} or Remark~\ref{rem.minimal}.
The \emph{Cremona dimension}, $\Crdim(G)$ is the minimal integer $n$ 
such that $G$ embeds into the Cremona group $\Cr(n)$ of birational
automorphisms of the affine space $\mathbb A^n$.

For the rest of this section we will assume that the base field $k$
is the field $\mathbb C$ of complex numbers.

\begin{conj} \label{conj:Dolgachev} (I.~Dolgachev, unpublished)
$\ed(G) \geqslant \Crdim (G)$ for every finite group $G$. 
\end{conj}

\begin{prop} \label{prop.incompatible}
\begin{enumerate}
\item
Conjecture~\ref{conj:GpVersalEnough}
implies $\ed( \PSL_2 (\bbF_{11})) = 3$. 

\item
Conjecture~\ref{conj:CSD}
implies $\ed( \PSL_2 (\bbF_{11})) = 3$. 

\item
Conjecture~\ref{conj:Dolgachev} implies $\ed(\PSL_2 (\bbF_{11})) = 4$. 
\end{enumerate}
\end{prop}

\begin{proof}
Consider the Klein cubic, i.e., the smooth cubic 
threefold $X \subset \bbP^4$ cut out by 
\[ x_0^2x_1 + x_1^2x_2 + x_2^2x_3 + x_3^2x_4 + x_4^2x_0 = 0 \, . \] 
The automorphism group of $X$ is $G= \PSL_2 (\bbF_{11})$.
The action of this group on $X$ is induced by a linear 
representation $\phi \colon G \to \GL_5$~\cite{adler}. 

It is easy to see that $3 \leqslant \ed( \PSL_2 (\bbF_{11})) \leqslant 4$. 
Indeed,
since $G$ is simple, the projection map $\bbC^5 \setminus \{ 0 \} \to \bbP^4$ 
shows that the $G$-action on $\bbP^4$ is generically free and very versal;
hence, $\ed( \PSL_2 (\bbF_{11})) \leqslant 4$.  The lower 
bound, $3 \leqslant \ed( \PSL_2 (\bbF_{11}))$,
follows from the fact that $\PSL_2 (\bbF_{11})$ cannot
act faithfully on a unirational surface; see~\cite{prokhorov, duncan2}.

Since $\phi$ is irreducible, no point of $X$ can be fixed by $G$.  
However, from~\cite{BeauvilleED} we see that $X$ has a $G_p$-fixed point
for any Sylow $p$-subgroup $G_p$ of $G$. 
Hence, by Corollary~\ref{cor.fixed-pt},
the $G_p$-action on $X$ is versal for every prime $p$. Now

\smallskip
(a) Conjecture~\ref{conj:GpVersalEnough}
implies that $X$ is $G$-versal. Thus
$\ed(G) \leqslant \dim(X) \leqslant 3$. 

\smallskip
(b) Conjecture~\ref{conj:CSD} also implies
that the $G$-action on $X$ is versal; see
Theorem~\ref{thm:cubicVersal}. Consequently, $\ed(G) \leqslant 3$,
as in part (a). 

\smallskip
(c) From work of Yu.~Prokhorov~\cite[Theorem~1.3]{prokhorov},
we see that there are no rational complex threefolds with a faithful 
action of $\PSL_2(\bbF_{11})$.  This means that
 $\Crdim (\PSL_2(\bbF_{11})) \geqslant 4$, 
and part (c) follows.
\end{proof}

We conclude that Conjectures~\ref{conj:CSD} and
\ref{conj:Dolgachev} are incompatible; they cannot both be true.
Same for Conjectures~\ref{conj:GpVersalEnough} and
~\ref{conj:Dolgachev}.

\begin{remark}
If we knew whether $\ed(\PSL_2 (\bbF_{11}))$ 
is $3$ or $4$, we would be able to complete the classification of finite 
simple groups of essential dimension $3$ over $\bbC$.
For details, see~\cite{BeauvilleED}.
\end{remark}

\begin{remark}
By~\cite[Theorem 1.5]{prokhorov},
there are only two equivariant birational equivalence classes of
rationally connected complex $\PSL_2(\bbF_{11})$-threefolds,
represented by the Klein cubic and a Fano threefold
of genus $8$.  These threefolds are (non-equivariantly) 
birationally equivalent~\cite[Remark 2.10]{prokhorov}. They 
are unirational (cf.~Theorem~\ref{thm:cubicUnirational}) but 
not rational~\cite{cg}.
\end{remark}

\begin{remark} \label{rem.quadric}
We are grateful to an anonymous referee for pointing out that
Theorem~\ref{thm:quadraticVersal} remains valid for an arbitrary
$G$-action on $X$, even if we do not assume that this action
comes from a linear representation $G \to \GL(V)$. 
Indeed, suppose $X \subset \bbP(V)$ 
is the zero locus of a non-degenerate quadratic form $q$ in $V$.
By~\cite[Corollary 69.6]{EKM} the natural inclusion
$\PGO(q) \hookrightarrow \Aut(X)$ is an isomorphism.
Thus the action of $G$ on $X$ gives rise to a morphism
$\psi \colon G \to \PGO(q)$, which fits into the diagram
\begin{equation} \label{e.rem-quadric}
 \xymatrix{ 1 \ar@{->}[r] & \bbG_m \ar@{->}[r]   
\ar@{=}[d] & G' \ar@{->}[r]  \ar@{->}[d]^{\phi} & 
G \ar@{->}[d]^{\psi} \ar@{->}[r] & 1 \cr 
1 \ar@{->}[r] & \bbG_m \ar@{->}[r]   
& \GO(q) \ar@{->}[r]  & \PGO(q) \ar@{->}[r] & 1 , }
\end{equation}
where $G'$ is the pull-back $G \times_{\PGO(q)} \GO(q)$.
Denote the connecting morphism
$H^1(K, \PGO(q)) \to H^2(K, \bbG_m)$
associated to the bottom sequence in~\eqref{e.rem-quadric}
by $\partial_K$. 
The connecting homomorphism for the top sequence is then 
given by the composition
\[ \partial_K^G \colon H^1(K, G) \xrightarrow{\psi_*} H^1(K, \PGO(q)) 
\xrightarrow{\partial_K} H^2(K, \bbG_m) \, , \] 
Now suppose $(T, K)$ is a twisting pair for $G$.
Then ${}^T X \subset {}^T \mathbb P(V)$, where 
${}^T \mathbb P(V)$ is a Brauer-Severi variety over $K$ whose class
is $\partial_K^G \, [T]$.
Here $[T]$ is the class of $T$ in $H^1(K, G)$
as in the proof of Proposition~\ref{prop:PnVersal}

We claim that, for any $\alpha \in H^1(K, \PGO(q))$,
the exponent of $\partial(\alpha) \in H^2(K, \bbG_m)$ is 
$1$ or $2$. Let us assume 
this claim for now. Taking $\alpha = \psi_* \, [T]$, 
we see that the exponent of $\partial_K^G \, [T]$ is $1$ or $2$.
Hence, the index of this class (or equivalently, the index of 
the Brauer-Severi variety ${}^T \mathbb P(V)$), is a power of $2$. 
If any of the conditions (a) - (f) of Theorem~\ref{thm:quadraticVersal}
hold, then ${}^T X$ (and thus ${}^T \mathbb P(V)$) has a point
or a $0$-cycle of odd degree over $K$. Hence, ${}^T \mathbb P(V)$ 
is split. Equivalently, $\partial_K^G \, [T] = 0$. Since this 
is true for every $T \in H^1(K, G)$ and every $K/k$, 
Proposition~\ref{prop:PnVersal} tells us that the top sequence
in~\eqref{e.rem-quadric} splits. Thus
$\psi \colon G \to \PGO(q)$ lifts to a linear 
representation $G \to \GO(q) \subset \GL(V)$, and we find ourselves 
in the setting of Theorem~\ref{thm:quadraticVersal}.

It remains to prove the claim. Let $n = \dim(V)$. If $n$ is odd, then 
by~\cite[Propositions 12.4 and 12.6]{BoI} the bottom short exact sequence
in~\eqref{e.rem-quadric} splits. Hence, in this case $\partial_K = 0$.
In other words, the exponent of $\partial_K(\alpha)$ is $1$ for every
$\alpha \in H^1(K, \PGO(q))$. 

Now assume that $n$ is even. In this case classes $\alpha \in H^1(K,\PGO(q))$ 
are in natural bijection with isomorphism classes 
of triples $(A, \sigma, f)$, where $A$ is a central 
simple algebra of degree $n$ and $(\sigma, f)$ is 
a quadratic pair on $A$; see~\cite[p. 409]{BoI}. 
The connecting map $\partial_K$ takes $\alpha$
to the Brauer class $[A] \in H^2(K, \bbG_m)$.  
By~\cite[Definition (5.4)]{BoI}, $\sigma$ is an involution of the first 
kind on $A$.  (Note that if $\Char(K) \ne 2$ then $\sigma$ is orthogonal 
and $f$ is uniquely determined by $\sigma$. In other words, 
in this case a quadratic pair on $A$ is the same thing 
as an orthogonal involution.) By a theorem 
of Albert (see, e.g.,~\cite[Theorem (3.1)]{BoI}), 
since $A$ has an involution of the first kind,
the exponent of $\partial_K(\alpha) = [A]$ is 
$1$ or $2$. This proves the claim.
\end{remark}

\begin{remark} \label{rem.quadric2}
Combining the argument of Remark~\ref{rem.quadric} with 
Lemma~\ref{lem:hilbertPoly}(c), we see that any action of
an algebraic group $G$ on a smooth irreducible quadric 
hypersurface $X \subset \mathbb{P}^n$ is $p$-versal, for 
every odd prime $p$.
\end{remark}

\subsection*{Acknowledgements}
The authors are grateful
to the anonymous referees for useful comments,
to J.-P.~Serre for allowing them to include his letter as an appendix,
and to
I.~Dolgachev, 
N.~Fakhruddin, 
R.~L{\"o}tscher, and
A.~Vistoli for helpful discussions.
The second author would like to thank the anonymous referee for~\cite{ckpr}
whose comments on versal actions served as a starting point 
for this project.

\section*{Appendix: Letter from J.-P.~Serre to Z.~Reichstein}

\hfill{Paris, June 10, 2010}

\medskip

   Dear Reichstein,

\medskip

   About ``versal'' :

\smallskip
There was first the notion of a ``universal object'', a notion which appeared
in several branches of mathematics around 1930--1950; there is even a
section of Bourbaki's {\em Th\'eorie des Ensembles} (chap.IV, $\S$2) 
on the general properties of this notion. An especially interesting 
case being the universal $G$-principal homogeneous space 
(now ``$G$-torsor''); the case of $G = \GL(n)$
was basically due to Chern. Such spaces ($E_G \to B_G$  was the standard
notation) were very useful to topologists; see e.g. Borel's thesis.

In the definition of ``universal'', there is a uniqueness property (up to
homotopy, sometimes) which is required. There are many interesting cases
where it does not hold (e.g. deformations of complex manifolds, \`a la
Kodaira- Spencer); people called them ``almost universal'' 
(or quasi , or  semi $\ldots$).  I do not know exactly when 
somebody had the amusing idea to call them ``versal'',
by deleting the ``uni'' which suggests uniqueness. I seem to remember
that it was Douady who did this (he enjoyed playing with words); the date
should be close to 1966, but I have not looked into his publications, and I
cannot give you a precise reference.\footnote{The earliest reference I
have been able to find is~\cite[p.~4]{douady}.  Z.R.}

That this idea applied to Galois cohomology was obvious from the beginning,
both to people with a topologist background (such as Rost or myself), and to
algebraists trying to parametrize equations (they rather used the word
``generic'', which I find a bit confusing.
\footnote{See Remark~\ref{rem.generic}. A.D.~and Z.R.})
But I don't think$^{(*)}$
the word ``versal'' got into print [in this context] before my UCLA lectures 
of 2001 
(do you know an earlier reference?)\footnote{The term {\em versal} 
is used in~\cite{br}, Section 7. The {\em versal polynomials} defined 
there give rise to versal $G$-torsors, in 
the sense of~\cite[Section I.5]{gms}, where $G$ is a finite group. Z.R.},  
even though I had used 
it in some College lectures around 1990 (especially those 
on ``negligible cohomology'', which were never written down).

  Note that the definition in UCLA has a rather non standard restriction:
it asks for a density property which may seem artificial (but it is
essential in Duncan's work!).

\smallskip
Best wishes,

J-P.Serre

\medskip
 (*) I have asked Google about ``versal torsor'', but all the references there
seem to be post 2001.

\end{document}